\documentclass[12pt,reqno]{amsart}
\usepackage{amsmath,amssymb,extarrows}
\usepackage{url}
\usepackage{tikz,enumerate}
\usepackage{diagbox}
\usepackage{appendix}
\usepackage{epic}
\usepackage{comment}

\usepackage{longtable} 

\usepackage{float} 

\usepackage{cite}

\usepackage{hyperref}
\usepackage{array}

\usepackage{booktabs}

\allowdisplaybreaks[4]

\usepackage{makecell}
\usepackage{float} 
\usepackage{siunitx} 
\usepackage{multirow}
\usepackage[figuresright]{rotating}

\newtheorem{theorem}{Theorem}
\newtheorem{corollary}[theorem]{Corollary}
\newtheorem{conj}[theorem]{Conjecture}
\newtheorem{lemma}[theorem]{Lemma}

\theoremstyle{definition}

\theoremstyle{remark}

\numberwithin{equation}{section}
\numberwithin{theorem}{section}
\numberwithin{defn}{section}

\begin{document}
\title[Nahm sum identities for Cartan matrices of type $D_k$]
 {Nahm sum identities for Cartan matrices of type $D_k$}

\author{Liuquan Wang and Shangwen Wang}

\address[L.\ Wang]{School of Mathematics and Statistics, Wuhan University, Wuhan 430072, Hubei, People's Republic of China}

\email{wanglq@whu.edu.cn;mathlqwang@163.com}

\address[S.\ Wang]{School of Mathematics and Statistics, Wuhan University, Wuhan 430072, Hubei, People's Republic of China}
\email{whuwsw@whu.edu.cn}

\subjclass[2010]{05A30, 11P84, 33D15, 33D60, 11F03}

\keywords{Nahm sums; Rogers--Ramanujan type identities; Bailey pairs; conformal field theory}

\begin{abstract}
Around 2007, Warnaar proved four identities related to Nahm sums associated with twice the inverse of the Cartan matrix of type $D_k$. Three of these had been conjectured by Flohr, Grabow, and Koehn, while special cases of two of the identities were first conjectured in 1993 by Kedem, Klassen, McCoy, and Melzer. Warnaar's proof relies on a multi-sum identity from Andrews' proof of the Andrews–Gordon identities. We give a new proof of all four identities using the theory of Bailey pairs. Furthermore, we establish a parametric generalization of two of the identities and provide two distinct proofs of this generalization.
\end{abstract}

\maketitle

\section{Introduction}\label{sec-intro}
As a famous and important topic linking the theory of $q$-series and modular forms, Nahm's problem \cite{Nahm1994,Nahm07} asks for finding all positive definite matrix $A\in \mathbb{Q}^{k\times k}$, $k$-dimensional column vector $B\in \mathbb{Q}^k$ and rational scalar $C$ such that the Nahm sum
    \begin{align}\label{eq-Nahm}
    f_{A,B,C}(q):=\sum_{n=(n_1,\dots,n_k)^\mathrm{T} \in \mathbb{N}^k}\frac{q^{\frac{1}{2}n^\mathrm{T}An+n^\mathrm{T}B+C}}{(q;q)_{n_1} \cdots (q;q)_{n_k}}
\end{align}
is modular. Here and throughout, we adopt the $q$-series notation: for  $n\in \mathbb{N}\cup\{\infty\}$,
\begin{align}
    (a;q)_n:=\prod\limits_{k=0}^{n-1}(1-aq^k)  \quad \text{and} \quad
    (a_1,a_2,\dots,a_m;q)_n:=\prod\limits_{k=1}^m(a_k;q)_n.
\end{align}
We call $f_{A,B,C}(q)$ the rank $k$ Nahm sum associated with $(A,B,C)$. Nahm's motivation comes from physics as such modular Nahm sums are usually characters of some 2-dimensional rational conformal field theories.

From the $q$-series perspective, Nahm's problem dates back to the famous Rogers--Ramanujan identities \cite{Rogers}:
\begin{align}
    \sum_{n=0}^\infty \frac{q^{n^2+\lambda n}}{(q;q)_n}=\frac{1}{(q^{\lambda+1},q^{4-\lambda};q^5)_\infty}, \quad \lambda=0,1.
\end{align}
This implies that for $k=1$, $f_{A,B,C}(q)$ is modular for $(A,B,C)=(2,0,1/60)$ and $(2,1,-11/60)$. These two identities have been generalized to the Andrews--Gordon identities \cite{Andrews1974} involving arbitrary odd moduli greater than 3: for integers $k,s$ such that $k\geq 1$ and $1\leq s \leq k+1$,
\begin{align}
\sum_{n_1,\dots,n_{k}\geq 0} \frac{q^{N_1^2+\cdots+N_{k}^2+N_s+\cdots +N_{k}}}{(q;q)_{n_1}(q;q)_{n_2}\cdots (q;q)_{n_{k}}}  =\frac{(q^s,q^{2k+3-s},q^{2k+3};q^{2k+3})_\infty}{(q;q)_\infty} \label{AG}
\end{align}
where
\begin{align}\label{eq-N-sum}
N_j=n_j+\cdots+n_{k} \quad (j\leq k), \quad \text{and} \quad N_{k+1}=0.
\end{align}
The Nahm sums involved here are associated with the matrix $2G_{k}$ where
\begin{align}\label{eq-matrix-G}
G_{k}:=(\min (i,j))_{1\leq i,j\leq k}.
\end{align}

In 2007, Zagier \cite{Zagier} solved Nahm's problem in the rank one case by proving that there are exactly seven modular rank one Nahm sums. He \cite[Tables 2 and 3]{Zagier} also provided some candidates for modular Nahm sums associated with 11 and 12 different matrices in the cases of rank two and three, respectively. After the works of Zagier \cite{Zagier}, Vlasenko--Zwegers \cite{VZ}, Wang \cite{Wang-rank2,Wang-rank3} and Cao--Rosengren--Wang \cite{CRW}, the modularity of these candidates have all been confirmed.

Zagier's lists \cite[Tables 2 and 3]{Zagier} do not exhaust all modular rank two and rank three Nahm sums. Some more modular Nahm sums in the rank two case were found by Vlasenko and Zwegers \cite{VZ}. Recently, Cao and Wang \cite{Cao-Wang2024,Cao-Wang2025} discovered some new modular Nahm sums in the rank three and rank four cases.

Inspired by the Andrews--Gordon identities \eqref{AG}, it is natural to ask whether we can find families of modular Nahm sums involving infinitely many ranks. A closely related identity can be found from Sembridge's 1990 work \cite[Corollary 1.5(b)]{Stembridge}:
\begin{align}
    \sum_{n_1,n_2,\dots,n_k\geq 0} \frac{q^{\frac{1}{2}(N_1^2+N_2^2+\cdots+N_k^2)}}{(q;q)_{n_1}(q;q)_{n_2}\cdots (q;q)_{n_k}}=\frac{(-q^{1/2};q)_\infty(q^{(k+1)/2},q^{(k+3)/2},q^{k+2};q^{k+2})_\infty}{(q;q)_\infty}.
\end{align}
The Nahm sum here is associated with the matrix $G_k$.
The even rank case $k=2m-2$ of this identity was conjectured by Melzer \cite{Melzer}, and Bressoud, Ismail and Stanton \cite[Theorem 5.1]{Bressoud2000} provided a different proof.  The odd rank case $k=2m-1$ was also proved by Warnaar \cite[Theorem 4.5]{Warnaar-2003}.

It should be noted that $G_k^{-1}$ is the Cartan matrix $\mathcal{C}(T_k)=(t_{ij})_{k\times k}$ of the tadpole Dynkin diagram with rank $k$ where
\begin{equation}\label{eq-matrix-Tk}
\begin{split}
   & t_{kk}=1, \quad t_{ii}=2, \quad 1\leq i \leq k-1, \quad t_{ij}=-1 \quad (|i-j|=1), \quad \text{and} \\
   & t_{ii}=0 \quad \text{otherwise}.
\end{split}
\end{equation}
Around 2016, Calinescu, Milas and Penn \cite[Conjecture 1]{CMP} conjectured that for any $k\geq 2$, $f_{\mathcal{C}(T_k),0,C}(q)$  is always modular for some scalar $C$ depending on $k$. This agrees with Zagier's duality expectation \cite[p.~50, (f)]{Zagier} which asserts that if the rank $k$ Nahm sum $f_{A,B,C}(q)$ is modular, then it is very likely that its dual Nahm sum $f_{A^\star,B^\star,C^\star}(q)$ is also modular where
\begin{align}\label{eq-dual}
    (A^\star, B^\star, C^\star)=(A^{-1},A^{-1}B,\frac{1}{2}B^\mathrm{T} A^{-1}B-\frac{k}{24}-C).
\end{align}
The case $k=2$ of the Calinescu--Milas--Penn conjecture has been proved in \cite{CMP}. The case $k=3$ was proved by Milas and Wang \cite{MW24} in 2024, and the cases $k=4,5$ were recently proved by Shi and Wang \cite{Shi-Wang}. However, the general case remains open and seems very complicated.

As for Cartan matrices of other type, around 1993, Kedem, Klassen, McCoy and Melzer \cite{Kedem} considered Nahm sums associated with twice the inverse Cartan matrices of the Lie algebras $A_n$, $D_n$, $E_6$, $E_7$ and $E_8$. In particular, they \cite[(2.9)]{Kedem}  discovered without proof that for $r=0,1$,
\begin{align}\label{eq-KKMM}
\sum_{\substack{n=(n_1,\dots,n_k)\in \mathbb{N}^k \\ n_{k-1}\equiv n_k +r \!\!\! \pmod{2}}}  \frac{q^{n^\mathrm{T}\mathcal{C}(D_k)^{-1}n}}{(q;q)_{n_1}\cdots (q;q)_{n_k}}=\frac{q^{\frac{1}{24}}}{\eta(\tau)}\sum_{n \in \mathbb{Z}} q^{k(n+\frac{1}{2}r)^2}.
\end{align}
Here $q=e^{2\pi i\tau}$ ($\mathrm{Im}~\tau>0$) and the Dedekind eta function is defined as
\begin{align}
\eta(\tau):=q^{\frac{1}{24}}(q;q)_\infty.
\end{align}
Recall the Cartan matrix $\mathcal{C}(D_k)=(d_{ij})_{k\times k}$ for the Lie algebra $D_k$ ($k\geq 3$):
\begin{equation}
\begin{split}
   & d_{ii}=2, \quad 1\leq i \leq k, \quad d_{i,j}=-1, \quad |i-j|=1,~~ 1\leq i,j\leq k-1, \\
   & d_{k-3,k}=d_{k,k-3}=-1, \quad \text{and} \quad d_{i,j}=0 \quad \text{otherwise}.
\end{split}
\end{equation}
The Nahm sum in \eqref{eq-KKMM} is associated with the matrix $A=(a_{ij})_{k\times k}=2\mathcal{C}(D_k)^{-1}$  with elements
\begin{equation}\label{eq-2D-inverse}
\begin{split}
&a_{k,k-1}=a_{k-1,k}=\frac{1}{2}k-1, \quad a_{k-1,k-1}=a_{k,k}=\frac{1}{2}k, \\
&a_{i,k-1}=a_{i,k}=a_{k-1,i}=a_{k,i}=i, \quad 1\leq i\leq k-2, \\
&a_{ij}=2\min(i,j), \quad 1\leq i,j\leq k-2.
\end{split}
\end{equation}
Here the matrix can be understood as $A=\mathcal{C}(A_1)\otimes \mathcal{C}(D_k)^{-1}$ as $\mathcal{C}(A_1)=2$. In fact,  many Nahm sums relevant to rational conformal field theory are associated with the matrix $A=G\otimes G'^{-1}$ where $G$ and $G'$ are ADET-type Cartan matrices. The matrix $2G_k$ mentioned above can also be understood as $\mathcal{C}(A_1)\otimes \mathcal{C}(T_k)^{-1}$.

As a generalization, Flohr, Grabow and Koehn \cite[Eqs.\ (2.28)--(2.30)]{FGK} conjectured the identities \eqref{id-1}, \eqref{id-3} and \eqref{id-4} below. Their conjecture was proved by Warnaar \cite{Warnaar}. Warnaar also proved an additional identity \eqref{id-2} that generalizes \eqref{eq-KKMM} in the case $r=1$. We state these results together as the following theorem.
\begin{theorem}\label{thm-main}
Let $k\geq 3$, $\lambda \in \mathbb{Q}$ and $c_{\lambda,k}=\frac{\lambda^2}{4k}-\frac{1}{24}$. We have
\begin{align}
&q^{c_{\lambda,k}}\sum_{\substack{n=(n_1,\dots,n_k)^\mathrm{T}\in \mathbb{N}^k \\ n_{k-1}\equiv n_k \!\!\! \pmod{2}}}  \frac{q^{n^\mathrm{T}\mathcal{C}(D_k)^{-1}n+\frac{\lambda}{2} (n_{k-1}-n_k)}}{(q;q)_{n_1}\cdots (q;q)_{n_k}}=\frac{1}{\eta(\tau)}\sum_{n\in \mathbb{Z}} q^{k(n+\frac{\lambda}{2k})^2}, \label{id-1} \\
&q^{c_{\lambda,k}}\sum_{\substack{n=(n_1,\dots,n_k)^\mathrm{T}\in \mathbb{N}^k \\ n_{k-1}\not\equiv n_k \!\!\! \pmod{2}}}  \frac{q^{n^\mathrm{T}\mathcal{C}(D_k)^{-1}n+\frac{\lambda}{2} (n_{k-1}-n_k)}}{(q;q)_{n_1}\cdots (q;q)_{n_k}}=\frac{1}{\eta(\tau)}\sum_{n\in \mathbb{Z}} q^{k(n+\frac{k-\lambda}{2k})^2}. \label{id-2}
\end{align}
Moreover, for $\lambda \in \{1,2,\dots,k-1\}$ we have
\begin{align}
&q^{c_{\lambda,k}}\sum_{\substack{n=(n_1,\dots,n_k)^\mathrm{T}\in \mathbb{N}^k \\ n_{k-1}\equiv n_k \!\!\! \pmod{2}}}  \frac{q^{n^\mathrm{T}\mathcal{C}(D_k)^{-1}n+\sum_{i=k-\lambda}^{k-2}(i-k+\lambda+1)n_i+\frac{\lambda}{2} (n_{k-1}+n_k)}}{(q;q)_{n_1}\cdots (q;q)_{n_k}} \nonumber \\
&=\frac{1}{\eta(\tau)}
\sum_{n\in \mathbb{Z}}(2n+1)q^{k(n+\frac{\lambda}{2k})^2}, \label{id-3} \\
&q^{c_{\lambda,k}}\sum_{\substack{n=(n_1,\dots,n_k)^\mathrm{T}\in \mathbb{N}^k \\ n_{k-1}\not\equiv n_k  \!\!\! \pmod{2}}}  \frac{q^{n^\mathrm{T}\mathcal{C}(D_k)^{-1}n+\sum_{i=k-\lambda}^{k-2}(i-k+\lambda+1)n_i+\frac{\lambda}{2} (n_{k-1}+n_k)}}{(q;q)_{n_1}\cdots (q;q)_{n_k}} \nonumber \\
&=\frac{2}{\eta(\tau)}
\sum_{n\in \mathbb{Z}}nq^{k(n-\frac{k-\lambda}{2k})^2}.  \label{id-4}
\end{align}
\end{theorem}
Note that in \cite{FGK} and \cite{Warnaar} the identities \eqref{id-1} and \eqref{id-2} were only stated for $\lambda\in \{0,1,\dots,k\}$, but we find that they hold for arbitrary $\lambda$. The right sides of \eqref{id-1}--\eqref{id-4} are bosonic character formulae corresponding to the $c_{k,1}$ logarithmic conformal field theory (see \cite{FGK} for details).

A key to Warnaar's proof is that the matrix $2\mathcal{C}(D_k)^{-1}$ contains the submatrix
$2G_{k-2}$ (see \eqref{eq-matrix-G} and \eqref{eq-2D-inverse}) which generates the Nahm sums in the Andrews--Gordon identities. This allows him to use Andrews' identity \cite[Eq.~(2.1)]{Andrews1974}
\begin{align}\label{eq-Andrews}
    &\sum_{n_1,\dots,n_{k-1}\geq 0} \frac{x^{N_1+\cdots +N_{k-1}}q^{N_1^2+\cdots +N_{k-i}^2+N_i+\cdots+N_{k-1}}}{(q;q)_{n_1}\cdots (q;q)_{n_{k-1}}} \nonumber \\
    &=\frac{1}{(xq;q)_\infty} \sum_{j=0}^\infty (-1)^jx^{kj}q^{\binom{j}{2}+kj^2+(k-i+1)j}(1-x^iq^{i(2j+1)})\frac{(xq;q)_j}{(q;q)_j}
\end{align}
to eliminate the summation variables $n_1,n_2,\dots,n_{k-2}$ simultaneously. Here $N_j$ is defined as in \eqref{eq-N-sum}.

This paper aims to provide a new proof for the above identities \eqref{id-1}--\eqref{id-4} and to give the following generalization of the identities \eqref{id-1} and \eqref{id-2}. We will not rely on \eqref{eq-Andrews} but proceed in a rather straightforward way using Bailey pairs.
\begin{theorem}\label{thm-gen-1}
Let $k\geq 2$, $a>\frac{1}{2}(k-1)$ and $\widetilde{A}=(\widetilde{a}_{ij})_{k\times k}$ be the matrix with elements
\begin{equation}\label{wA-defn}
\begin{split}
\widetilde{a}_{k,k-1}=\widetilde{a}_{k-1,k}=k-1-a, \quad \widetilde{a}_{k-1,k-1}=\widetilde{a}_{k,k}=a, \\
\widetilde{a}_{i,k-1}=\widetilde{a}_{i,k}=\widetilde{a}_{k-1,i}=\widetilde{a}_{k,i}=i, \quad 1\leq i\leq k-2, \\
\widetilde{a}_{ij}=2\min(i,j), \quad 1\leq i,j\leq k-2.
\end{split}
\end{equation}
Then $\widetilde{A}$ is positive and we have
\begin{align}
    \sum_{\substack{n=(n_1,\dots,n_k)^\mathrm{T}\in \mathbb{N}^k \\ n_{k-1}\equiv n_k \!\!\pmod{2}}}  \frac{q^{\frac{1}{2}n^\mathrm{T}\widetilde{A}n}x^{n_{k-1}-n_k}}{(q;q)_{n_1}\cdots (q;q)_{n_k}}=\frac{(-q^{2a}x^2,-q^{2a}x^{-2},q^{4a};q^{4a})_\infty}{(q;q)_\infty}, \label{ax-even}\\
     \sum_{\substack{n=(n_1,\dots,n_k)^\mathrm{T}\in \mathbb{N}^k \\ n_{k-1}\not\equiv n_k \!\!\pmod{2}}}  \frac{q^{\frac{1}{2}n^\mathrm{T}\widetilde{A}n}x^{n_{k-1}-n_k}}{(q;q)_{n_1}\cdots (q;q)_{n_k}}=q^{\frac{1}{2}a}x\frac{(-q^{4a}x^2,-x^{-2},q^{4a};q^{4a})_\infty}{(q;q)_\infty} \label{ax-odd}
\end{align}
and
\begin{align}\label{eq-thm-ax}
    \sum_{n=(n_1,\dots,n_k)^\mathrm{T}\in \mathbb{N}^k}  \frac{q^{\frac{1}{2}n^\mathrm{T}\widetilde{A}n}x^{n_{k-1}-n_k}}{(q;q)_{n_1}\cdots (q;q)_{n_k}}=\frac{(-xq^{a/2},-q^{a/2}x^{-1},q^a;q^a)_\infty}{(q;q)_\infty}.
\end{align}
\end{theorem}

Setting $a=\frac{1}{2}k$ and $x=q^{\frac{\lambda}{2}}$ in \eqref{ax-even} and \eqref{ax-odd}, we obtain \eqref{id-1} and \eqref{id-2}, respectively. Note that when $k=2$ and $k=3$, the matrix in Theorem \ref{thm-gen-1} is
 \begin{align*}
   \widetilde{A}_2=\begin{pmatrix} a & 1-a \\ 1-a & a \end{pmatrix} \quad \text{and} \quad   \widetilde{A}_3= \left(\begin{array}{ccc}
            2 & 1 & 1 \\
            1 & a & 2-a \\
            1 & 2-a & a
        \end{array} \right),
   \end{align*}
respectively. The matrix $\widetilde{A}_2$ is Zagier's first example in his rank two list \cite[Table 2]{Zagier}. Zagier \cite[Eq.~(26)]{Zagier} proved that
\begin{align}\label{eq-Zagier-rank2}
    \sum_{n=(n_1,n_2)^\mathrm{T}\in \mathrm{N}^2}\frac{q^{\frac{1}{2}n^\mathrm{T}\widetilde{A}_2n+\frac{\lambda}{2}(n_1-n_2)}}{(q;q)_{n_1}(q;q)_{n_2}}=\frac{1}{(q;q)_\infty}\sum_{n\in \mathbb{Z}}q^{\frac{1}{2}an^2+\frac{\lambda}{2}n}.
\end{align}
This corresponds to the special instance $(k,x)=(2,q^{\frac{\lambda}{2}})$ of \eqref{eq-thm-ax}.
The second matrix $\widetilde{A}_3$ appeared in the work of Cao and Wang \cite{Cao-Wang2024} which was lifted from $\widetilde{A}_2$ through the identity  (see \cite[p.~20]{Andrews1981} or \cite[Eq.~(13)]{Zagier}):
\begin{align}\label{eq-lift}
    \frac{1}{(q;q)_i(q;q)_j}=\sum_{\ell \geq 0} \frac{q^{(i-\ell)(j-\ell)}}{(q;q)_\ell(q;q)_{i-\ell}(q;q)_{j-\ell}}.
\end{align}
Substituting \eqref{eq-lift} into \eqref{eq-Zagier-rank2} with $(i,j,\ell)=(n_1,n_2,n_3)$, replacing $n_i$ by $n_i+n_3$ ($i=1,2$) and then replacing $(n_1,n_2,n_3)$ by $(n_2,n_3,n_1)$, we obtain (see also \cite[(3.2) and (1.8)]{Cao-Wang2024})
\begin{align}\label{eq-intro-Z}
    \sum_{\substack{n=(n_1,n_2,n_3)^\mathrm{T}\in \mathbb{N}^3}} \frac{q^{\frac{1}{2}n^\mathrm{T}\widetilde{A}_3n+\frac{\lambda}{2}(n_2-n_3)}}{(q;q)_{n_1}(q;q)_{n_2}(q;q)_{n_3}}=\frac{1}{(q;q)_\infty}\sum_{n\in \mathbb{Z}} q^{\frac{1}{2}an^2+\frac{\lambda}{2} n}.
\end{align}
This process is called a lifting operation in \cite{Cao-Wang2024}. The identity \eqref{eq-intro-Z} corresponds to the special case $(k,x)=(3,q^{\frac{\lambda}{2}})$ of \eqref{eq-thm-ax}. The case $a=2$ of \eqref{eq-intro-Z} was conjectured by Zagier \cite[Eq.~(33)]{Zagier} corresponding to the seventh example in his rank three list \cite[Table 3]{Zagier}, and was first proved by Wang \cite{Wang-rank3}. Wang \cite[(4.49)]{Wang-rank3} also stated the stronger but essentially equivalent form of \eqref{eq-intro-Z}:
\begin{align}\label{eq-Wang-id}
    \sum_{n_1,n_2,n_3\geq 0} \frac{q^{n_1^2+n_2^2+n_3^2+n_1(n_2+n_3)}x^{n_2-n_3}}{(q;q)_{n_1}(q;q)_{n_2}(q;q)_{n_3}}=\frac{(-qx,-qx^{-1},q^2;q^2)_\infty}{(q;q)_\infty}.
\end{align}
As such, the identity \eqref{eq-thm-ax} can be regarded as a generalization of \eqref{eq-Zagier-rank2}, \eqref{eq-intro-Z} and \eqref{eq-Wang-id}. It should be emphasized that the proof of \eqref{eq-intro-Z} (equivalently \eqref{eq-Wang-id}) in \cite{Wang-rank3} is based on the constant term method and does not use Bailey pairs. It seems difficult to apply the method there to prove the general identity \eqref{eq-thm-ax}.

The above discussion of special instances of \eqref{eq-thm-ax} eventually leads us to find another direct proof for Theorem \ref{thm-gen-1}. This short proof does not use any Bailey pairs. Instead, we will show that the identity \eqref{eq-thm-ax} can be deduced from the  case $k=2$ through repetitive use of \eqref{eq-lift}.

The remainder of this paper is organized as follows. In Section \ref{sec-pre} we recall some basic knowledge about Bailey pairs. In Section \ref{sec-proof} we first give two different proofs for Theorem \ref{thm-gen-1}, and then present our new proof for Theorem \ref{thm-main}. Finally, in Section \ref{sec-remark} we present an interesting consequence of Theorem \ref{thm-gen-1} and briefly discuss the Nahm sums related to the matrices $\frac{1}{2}\mathcal{C}(D_k)$, $\mathcal{C}(D_k)$ and $\mathcal{C}(D_k)^{-1}$.

\section{Preliminaries}\label{sec-pre}
Recall the Jacobi triple product identity \cite[Theorem 2.8]{Andrews1998}
\begin{align}\label{jtp}
    \sum_{n=-\infty}^\infty x^nq^{n^2}=(-qx,-q/x,q^2;q^2)_\infty, \quad x\neq 0.
\end{align}
A pair of sequences $(\alpha_n(a;q),\beta_n(a;q))$ is called a Bailey pair relative to the parameter $a$ if for all $n\geq 0$,
 \begin{align}\label{defn-BP}
     \beta_n(a;q)=\sum_{k=0}^n\frac{\alpha_k(a;q)}{(q;q)_{n-k}(aq;q)_{n+k}}.
 \end{align}
As direct consequences of the Bailey lemma (see, e.g.~ \cite[Theorem 2.1]{Warnaar-2001}), we obtain a Bailey pair $(\alpha_n',\beta_n')$ from a Bailey pair $(\alpha_n,\beta_n)$ through the formula (see, e.g.~\cite[Eq.\ (S1)]{Bressoud2000} or \cite[(12)]{Warnaar-2001}):
\begin{align}\label{BP-S1}
\alpha_n'(a;q)=a^nq^{n^2}\alpha_n(a;q), \quad \beta_n'(a;q)=\sum_{r=0}^n \frac{a^rq^{r^2}}{(q;q)_{n-r}}\beta_r(a;q).
\end{align}
When $n\rightarrow \infty$, we deduce from \eqref{defn-BP} and \eqref{BP-S1} that
\begin{align}\label{eq-BP-id-key}
\sum_{n=0}^\infty a^nq^{n^2}\beta_n(a;q)=\frac{1}{(aq;q)_\infty} \sum_{n=0}^\infty a^n q^{n^2}\alpha_n(a;q).
\end{align}

The following lemma allows us to change the parameter of the Bailey pair.
\begin{lemma}\label{lem-BP-change}
If $(\alpha_n(a;q),\beta_n(a;q))$ is a Bailey pair relative to $a$, then $(\alpha_n',\beta_n')$ is a Bailey pair relative to $aq$:
\begin{equation}\label{eq-BP-lift}
\begin{split}
&\alpha_n'(aq;q)=\frac{(1-aq^{2n+1})a^nq^{n^2}}{1-aq}\sum_{r=0}^n a^{-r}q^{-r^2}\alpha_r(a;q), \\
&\beta_n'(aq;q)=\beta_n(a;q).
\end{split}
\end{equation}
Moreover, $(\widetilde{\alpha}_n,\widetilde{\beta}_n)$ is a Bailey pair relative to $a/q$ where
\begin{align}\label{eq-BP-reduce}
\widetilde{\alpha}_0(a/q;q)&=\alpha_0(a;q), \quad \widetilde{\alpha}_n(a/q;q)=(1-a)\Big(\frac{\alpha_n(a;q)}{1-aq^{2n}}-\frac{aq^{2n-2}\alpha_{n-1}(a;q)}{1-aq^{2n-2}}   \Big), \nonumber \\
\widetilde{\beta}_n(a/q;q)&=\beta_n(a;q).
\end{align}
\end{lemma}
The first assertion \eqref{eq-BP-lift} follows from a transformation formula in Lovejoy's work \cite[(2.4) and (2.5)]{Lovejoy2004} (see also \cite[(2.40)]{Cao-Wang2024}).
The second assertion \eqref{eq-BP-reduce} is a special instance of a general theorem of Warnaar \cite[Theorem 3.2]{Warnaar-2001} (see also \cite[Lemma 2.3]{Cao-Wang2024}).

\section{Proofs of the theorems}\label{sec-proof}

\begin{proof}[First proof of Theorem \ref{thm-gen-1}]
Note that
\begin{align*}
\widetilde{A}=\begin{pmatrix}
2G_{k-2} & P \\
P^\mathrm{T} & Q
\end{pmatrix}, ~~ P^\mathrm{T}=\begin{pmatrix}
    1 & 2 & \cdots & k-2 \\
    1 & 2 & \cdots & k-2
\end{pmatrix}, ~~ Q=\begin{pmatrix}
    a & k-1-a \\ k-1-a & a
\end{pmatrix}.
\end{align*}
It is already known that $2G_{k-2}$ is positive. To prove that $\widetilde{A}$ is positive, it suffices to show that
$\Delta:=Q-P^\mathrm{T}(2G_{k-2})^{-1}P$ (the Schur complement of $2G_{2k-2}$ in $\widetilde{A}$)
is positive. In fact, by direct calculations using \eqref{eq-matrix-Tk} we find that
\begin{align}
\Delta=\begin{pmatrix}
    a-\frac{1}{2}k+1 & \frac{1}{2}k-a \\
    \frac{1}{2}k-a & a -\frac{1}{2}k+1
\end{pmatrix}.
\end{align}
The eigenvalues of $\Delta$ are $1$ and $2a-k+1$, respectively. Hence $\Delta$ is positive if and only if $a>\frac{1}{2}(k-1)$.

We denote the left side of \eqref{ax-odd} and \eqref{ax-even} by $S_0(q)$ and $S_1(q)$, respectively.  Note that
\begin{align}\label{eq-quadratic}
&\frac{1}{2}n^\mathrm{T}\widetilde{A}n=\sum_{i=1}^{k-2} in_i^2+\frac{a}{2}(n_{k-1}^2+n_k^2)+2\sum_{1\leq i <j \leq k-2}in_in_j  \nonumber \\
&\quad \quad +\sum_{i=1}^{k-2}in_i(n_{k-1}+n_k) +(k-1-a)n_{k-1} n_k \nonumber \\
&=\big(n_1+n_2+\cdots+n_{k-2}+\frac{n_{k-1}+n_k}{2}\big)^2+\cdots+\big(n_{k-2}+\frac{n_{k-1}+n_k}{2}\big)^2 \nonumber \\
&\quad \quad +\frac{2a+2-k}{4}(n_{k-1}^2+n_k^2) + \frac{k-2a}{2}n_{k-1 } n_k.
\end{align}

According to the parity of $n_{k-1}+n_k$, we divide our discussions into two cases.

\textbf{Case 1.} If $n_{k-1}+n_k$ is even, then we write
\begin{align}\label{nk-even}
n_{k-1}=s_{k-1}+r, ~ n_k=s_{k-1}-r, ~ -s_{k-1}\leq r \leq s_{k-1}, ~s_{k-1}\in \mathbb{N},
\end{align}
and we introduce new variables:
\begin{equation}\label{eq-variable}
\begin{split}
n_1+n_2+\cdots +n_{k-2}+s_{k-1}&=s_1, \\
n_2+\cdots+n_{k-2}+s_{k-1}&=s_2, \\
\cdots  &\\
n_{k-2}+s_{k-1}&=s_{k-2}.
\end{split}
\end{equation}
In other words, we have
\begin{align}\label{eq-variable-relation}
n_1=s_1-s_2,~ n_2=s_2-s_3, ~\cdots,~  n_{k-2}=s_{k-2}-s_{k-1}.
\end{align}

We have
\begin{align}
&{S}_0(q)=\sum_{s_1\geq s_2\geq \cdots \geq s_{k-2}\geq s_{k-1}\geq 0} \frac{q^{s_1^2+s_2^2+\cdots +s_{k-2}^2+s_{k-1}^2}}{(q;q)_{s_1-s_2}(q;q)_{s_2-s_3}\cdots (q;q)_{s_{k-3}-s_{k-2}}(q;q)_{s_{k-2}-s_{k-1}}} \nonumber \\
&\quad \quad \times \sum_{r=-s_{k-1}}^{s_{k-1}} \frac{q^{(2a+1-k)r^2}x^{2r}}{(q;q)_{s_{k-1}+r}(q;q)_{s_{k-1}-r}} \nonumber \\
&=\sum_{s_1\geq s_2\geq \cdots \geq s_{k-2}\geq s_{k-1}\geq 0} \frac{q^{s_1^2+s_2^2+\cdots +s_{k-2}^2+s_{k-1}^2}}{(q;q)_{s_1-s_2}(q;q)_{s_2-s_3}\cdots (q;q)_{s_{k-3}-s_{k-2}}(q;q)_{s_{k-2}-s_{k-1}}} \beta_{s_{k-1}}^{(1)}(1;q), \label{proof-G1-start}
\end{align}
where $(\alpha_r^{(1)}(1;q),\beta_r^{(1)}(1;q))$ is the  Bailey pair relative to 1 with
\begin{align}
\alpha_r^{(1)}(1;q)=\left\{\begin{array}{ll}
1, & r=0, \\
q^{(2a+1-k)r^2}(x^{2r}+x^{-2r}), & r\geq 1.
\end{array}\right.
\end{align}
Substituting this Bailey pair into \eqref{BP-S1} and iterating $i$ times, we obtain the Bailey pairs $(\alpha_n^{(1+i)}(1;q),\beta_n^{(1+i)}(1;q)$ with
\begin{align}
\alpha_n^{(1+i)}(1;q)=q^{in^2}\alpha_n^{(1)}(1;q), \quad i=1,2,\dots,k-2.
\end{align}
Substituting these Bailey pairs into \eqref{proof-G1-start}, we deduce that
\begin{align}\label{proof-S0}
&{S}_0(q)=\sum_{s_1\geq s_2\geq \cdots \geq s_{k-2}\geq 0} \frac{q^{s_1^2+s_2^2+\cdots +s_{k-2}^2}}{(q;q)_{s_1-s_2}(q;q)_{s_2-s_3}\cdots (q;q)_{s_{k-3}-s_{k-2}}}\beta_{s_{k-2}}^{(2)}(1;q) \nonumber \\
&=\cdots =\sum_{s_1\geq 0} q^{s_1^2}\beta_{s_1}^{(k-1)}(1;q) \nonumber \\
&=\frac{1}{(q;q)_\infty} \sum_{n=0}^\infty q^{n^2}\alpha_n^{(k-1)}(1;q) \quad \text{(by \eqref{eq-BP-id-key})} \nonumber \\
&=\frac{1}{(q;q)_\infty}\Big(1+\sum_{n=1}^\infty q^{2an^2}(x^{2n}+x^{-2n})\Big)=\frac{1}{(q;q)_\infty} \sum_{n=-\infty}^\infty q^{2an^2}x^{2n}.
\end{align}
This proves \eqref{ax-odd} upon using \eqref{jtp}.

\textbf{Case 2.} If $n_{k-1}\not\equiv n_k$ (mod 2), then we write
\begin{align}
\label{odd-nk}
n_{k-1}=s_{k-1}+r+1, ~ n_k=s_{k-1}-r, ~ -s_{k-1}-1\leq r \leq s_{k-1},~ s_{k-1}\in \mathbb{N},
\end{align}
and we  introduce the variables $s_1,s_2,\dots,s_{k-2}$ as in \eqref{eq-variable} and \eqref{eq-variable-relation}.

 Using \eqref{odd-nk} and \eqref{eq-variable} we have
\begin{align}\label{G2-proof-start}
&{S}_1(q)=q^{\frac{1}{2}a}\sum_{s_1\geq s_2\geq \cdots \geq s_{k-2}\geq s_{k-1}\geq 0} \frac{q^{s_1^2+\cdots +s_{k-2}^2+s_{k-1}^2+s_1+s_2+\cdots +s_{k-2}+s_{k-1}}}{(q;q)_{s_1-s_2}\cdots (q;q)_{s_{k-2}-s_{k-1}}} \nonumber \\
 &\quad \quad \times \sum_{r=-s_{k-1}-1}^{s_{k-1}} \frac{q^{(2a+1-k)(r^2+r)}x^{2r+1}}{(q;q)_{s_{k-1}-r}(q;q)_{s_{k-1}+r+1}} \nonumber \\
&=q^{\frac{1}{2}a}\sum_{s_1\geq s_2\geq \cdots \geq s_{k-2}\geq s_{k-1}\geq 0} \frac{q^{s_1^2+\cdots +s_{k-2}^2+s_{k-1}^2+s_1+s_2+\cdots +s_{k-2}+s_{k-1}}}{(q;q)_{s_1-s_2}\cdots (q;q)_{s_{k-2}-s_{k-1}}}  \nonumber \\
&\quad \quad \times \sum_{r=0}^{s_{k-1}} \Big(\frac{q^{(2a+1-k)(r^2+r)}x^{2r+1}}{(q;q)_{s_{k-1}-r}(q;q)_{s_{k-1}+r+1}}+\frac{q^{(2a+1-k)(r^2+r)}x^{-2r-1}}{(q;q)_{s_{k-1}-r}(q;q)_{s_{k-1}+r+1}}  \Big) \nonumber \\
&=\frac{q^{\frac{1}{2}a}}{1-q}\sum_{s_1\geq s_2\geq \cdots \geq s_{k-2}\geq s_{k-1}\geq 0} \frac{q^{s_1^2+\cdots +s_{k-2}^2+s_{k-1}^2+s_1+s_2+\cdots +s_{k-2}+s_{k-1}}}{(q;q)_{s_1-s_2}\cdots (q;q)_{s_{k-2}-s_{k-1}}}\beta_{s_{k-1}}^{(1)}(q;q).
\end{align}
Here $(\alpha_r^{(1)}(q;q),\beta_r^{(1)}(q;q))$ is the Bailey pair relative to $q$ with
\begin{align}\label{G2-alpha-pre}
\alpha_r^{(1)}(q;q)=q^{(2a+1-k)(r^2+r)}(x^{2r+1}+x^{-2r-1}), \quad r\geq 0.
\end{align}
Applying \eqref{BP-S1} to this Bailey pair and iterating $i$ times, we obtain the Bailey pairs
$(\alpha_n^{(1+i)}(q;q),\beta_n^{(1+i)}(q;q))$ with
\begin{align}\label{G2-alpha-mid}
\alpha_n^{(1+i)}(q;q)=q^{i(n^2+n)}\alpha_n^{(1)}(q;q), \quad i=1,2,\dots,k-2.
\end{align}
Substituting these Bailey pairs into \eqref{G2-proof-start}, we deduce that
\begin{align}\label{proof-S1}
&{S}_1(q)=\frac{q^{\frac{1}{2}a}}{1-q}\sum_{s_1\geq s_2\geq \cdots \geq s_{k-2}\geq 0} \frac{q^{s_1^2+\cdots +s_{k-2}^2+s_1+s_2+\cdots +s_{k-2}}}{(q;q)_{s_1-s_2}\cdots (q;q)_{s_{k-3}-s_{k-2}}}\beta_{s_{k-2}}^{(2)}(q;q)\nonumber \\
&=\cdots =\frac{q^{\frac{1}{2}a}}{1-q}\sum_{s_1\geq 0} q^{s_1^2+s_1}\beta_{s_1}^{(k-1)}(q;q)\nonumber \\
&=\frac{q^{\frac{1}{2}a}}{(1-q)(q^2;q)_\infty}\sum_{n=0}^\infty q^{n^2+n}\alpha_n^{(k-1)}(q;q) \quad \text{(by \eqref{eq-BP-id-key})} \nonumber \\
&=\frac{q^{\frac{1}{2}a}}{(q;q)_\infty}\sum_{n=0}^\infty q^{2a(n^2+n)}(x^{2n+1}+x^{-2n-1}) \nonumber \\
&=\frac{q^{\frac{1}{2}a}}{(q;q)_\infty}\sum_{n=-\infty}^\infty q^{2a(n^2+n)}x^{2n+1}.
\end{align}
Here for the penultimate line we used \eqref{G2-alpha-pre} and \eqref{G2-alpha-mid}, and we replaced $n$ by $-n-1$ for the second sum to obtain the last equality.
This proves \eqref{ax-even} upon using \eqref{jtp}.

Adding \eqref{proof-S0} and \eqref{proof-S1} together, we deduce that
\begin{align}
    S_0(q)+S_1(q)=\frac{1}{(q;q)_\infty}\sum_{n=-\infty}^\infty q^{\frac{1}{2}an^2}x^n.
\end{align}
This proves \eqref{eq-thm-ax} upon using \eqref{jtp}.
\end{proof}

We now provide another proof for Theorem \ref{thm-gen-1}, which reveals that  \eqref{eq-thm-ax} can be deduced from \eqref{eq-Zagier-rank2}.
\begin{proof}[Second proof of Theorem \ref{thm-gen-1}]
Note that \eqref{ax-even} and \eqref{ax-odd} can be obtained  by extracting the even and odd  powers of $x$ from both sides of \eqref{eq-thm-ax}. Therefore, it suffices to prove \eqref{eq-thm-ax}.

For $n=(n_1,n_2,\dots,n_k)^\mathrm{T}$ we denote $Q_k(n_1,n_2,\dots,n_k)=\frac{1}{2}n^\mathrm{T}\widetilde{A}n$. Replacing $(n_1,n_2,\dots,n_k)$ by $(n_3,n_4,\dots,n_k,n_1,n_2)$, the identity \eqref{eq-thm-ax} is equivalent to
\begin{align}\label{eq-thm-ax-equivalent}
      F_k(q):=\sum_{n_1,\dots,n_k\geq  0}  \frac{q^{Q_k(n_3,n_4,\dots,n_k,n_1,n_2)}x^{n_{1}-n_2}}{(q;q)_{n_1}\cdots (q;q)_{n_k}}=\frac{(-xq^{a/2},-q^{a/2}x^{-1},q^a;q^a)_\infty}{(q;q)_\infty}.
\end{align}
Substituting \eqref{eq-lift} into \eqref{eq-thm-ax-equivalent} with $(i,j,\ell)=(n_1,n_2,n_{k+1})$, and then replacing $n_i$ by $n_i+n_{k+1}$ ($i=1,2$), we deduce that
\begin{align}\label{proof-Fk}
    F_k(q)=\sum_{n_1,\dots,n_k,n_{k+1}\geq 0} \frac{q^{Q_k(n_3,\dots,n_k,n_1+n_{k+1},n_2+n_{k+1})+n_1n_2}x^{n_{1}-n_2}}{(q;q)_{n_1}\cdots (q;q)_{n_{k-1}}(q;q)_{n_{k}}(q;q)_{n_{k+1}}}.
\end{align}
From \eqref{eq-quadratic} we deduce that
\begin{align}
&Q_k(n_3,\dots,n_k,n_1+n_{k+1},n_2+n_{k+1})+n_1n_2\nonumber \\
&=\Big(\frac{n_1+n_2+2n_{k+1}}{2}+n_3+n_4+\cdots+n_{k}\Big)^2+\cdots+\Big(\frac{n_{1}+n_2+2n_{k+1}}{2}+n_k\Big)^2 \nonumber \\
&\quad \quad +\frac{2a+2-k}{4}\Big(\big(n_1+n_{k+1}\big)^2+\big(n_2+n_{k+1}\big)^2\Big) \nonumber \\
&\quad \quad + \frac{k-2a}{2}\big(n_1+n_{k+1}\big)\big(n_2+n_{k+1}\big)+n_1n_2 \nonumber \\
&=\big(\frac{n_1+n_2}{2}+n_3+n_4+\cdots+n_{k}+n_{k+1}\big)^2+\cdots+\big(\frac{n_{1}+n_2}{2}+n_k+n_{k+1}\big)^2 \nonumber \\
&\quad \quad +\Big(\frac{n_1+n_2}{2}+n_{k+1}\Big)^2+\frac{2a+1-k}{4}(n_1^2+n_2^2)+\frac{k+1-2a}{2}n_1n_2 \nonumber \\
&=Q_{k+1}(n_3,n_4,\dots,n_k,n_{k+1},n_1,n_2).
\end{align}
Substituting it into \eqref{proof-Fk}, we conclude that
\begin{align}\label{Fk-rec}
    F_k(q)=F_{k+1}(q), \quad k\geq 2.
\end{align}
It remains to prove \eqref{eq-thm-ax-equivalent} for $k=2$ which is known and equivalent to Zagier's identity \eqref{eq-Zagier-rank2}. For the sake of completeness, we reproduce the proof. Recall the Durfee rectangle identity (see \cite[Corollary 2.6]{Andrews1998} or \cite[Eq.~(27)]{Zagier}):
\begin{align}
\sum_{n_1-n_2=n}\frac{q^{n_1n_2}}{(q;q)_{n_1}(q;q)_{n_2}}=\frac{1}{(q;q)_\infty}.
\end{align}
Utilizing it we deduce that
\begin{align*}
    &F_2(q)=\sum_{n_1,n_2\geq 0} \frac{q^{\frac{a}{2}(n_1-n_2)^2+n_1n_2}x^{n_1-n_2}}{(q;q)_{n_1}(q;q)_{n_2}} \nonumber \\
    &=\sum_{n=-\infty}^\infty q^{\frac{a}{2}n^2}x^n \sum_{n_1-n_2=n}\frac{q^{n_1n_2}}{(q;q)_{n_1}(q;q)_{n_2}}=\frac{1}{(q;q)_\infty}\sum_{n=-\infty}^\infty q^{\frac{a}{2}n^2}x^n. \qedhere
\end{align*}
\end{proof}

If we expand the right side of \eqref{eq-thm-ax} using \eqref{jtp} and then compare the coefficients of $x^N$ on both sides, we obtain the following interesting consequence.
\begin{corollary}
Let $k\geq 2$, $a>0$ and $\widetilde{A}$ be the matrix in \eqref{wA-defn}. For any  $N\in \mathbb{Z}$ we have
\begin{align}\label{eq-cor}
   \sum_{\substack{n=(n_1,\dots,n_k)^\mathrm{T}\in \mathbb{N}^k \\ n_{k-1}-n_k=N}}  \frac{q^{\frac{1}{2}n^\mathrm{T}\widetilde{A}n}}{(q;q)_{n_1}\cdots (q;q)_{n_k}} =\frac{q^{\frac{1}{2}aN^2}}{(q;q)_\infty}.
\end{align}
\end{corollary}
The case $k=3$  of \eqref{eq-cor} was pointed out in \cite[(4.50) and (4.51)]{Wang-rank3} as a consequence of \eqref{eq-Wang-id}, and there the second author asked ``if one can prove this fact directly''. Clearly, a slight modification of the second proof above serves as an answer to this question.

We have seen from the second proof that Bailey pairs are not necessary to prove Theorem \ref{thm-gen-1}. Nevertheless, the first proof still has some advantage in the sense that similar arguments will enable us to prove \eqref{id-3} and \eqref{id-4} as well.

\begin{proof}[Proof of Theorem \ref{thm-main}]
The identities \eqref{id-1} and \eqref{id-2} follow from Theorem \ref{thm-gen-1} with $a=\frac{1}{2}k$.

In order to prove \eqref{id-3}, we denote its left side as $T_1(q)$ and introduce the variables $s_1,s_2,\dots,s_{k-1}$ and $r$ as in \eqref{nk-even}--\eqref{eq-variable-relation}. From \eqref{eq-variable-relation} we deduce that
\begin{align}\label{index-simplify}
&\sum_{i=k-\lambda}^{k-2} (i-k+\lambda+1)n_i=\sum_{i=k-\lambda}^{k-2}(i-k+\lambda+1)(s_i-s_{i+1}) \nonumber \\
&=\sum_{i=k-\lambda}^{k-2} s_i-(\lambda-1)s_{k-1}.
\end{align}
Using \eqref{eq-quadratic} and \eqref{index-simplify} we have
\begin{align}\label{S2-proof-start}
&q^{-c_{\lambda,k}}T_1(q)=\sum_{s_1\geq s_2\geq \cdots \geq s_{k-2}\geq s_{k-1}\geq 0} \frac{q^{s_1^2+\cdots +s_{k-2}^2+s_{k-1}^2+s_{k-\lambda}+\cdots +s_{k-2}+s_{k-1}}}{(q;q)_{s_1-s_2}\cdots (q;q)_{s_{k-3}-s_{k-2}}(q;q)_{s_{k-2}-s_{k-1}}} \nonumber \\
&\quad \times \sum_{r=-s_{k-1}}^{s_{k-1}} \frac{q^{r^2}}{(q;q)_{s_{k-1}+r}(q;q)_{s_{k-1}-r}} \nonumber \\
&=\sum_{s_1\geq s_2\geq \cdots \geq s_{k-2}\geq s_{k-1}\geq 0} \frac{q^{s_1^2+\cdots +s_{k-2}^2+s_{k-1}^2+s_{k-\lambda}+\cdots +s_{k-2}+s_{k-1}}}{(q;q)_{s_1-s_2}\cdots (q;q)_{s_{k-3}-s_{k-2}}(q;q)_{s_{k-2}-s_{k-1}}} \beta_{s_{k-1}}^{(1)}(1;q),
\end{align}
where $(\alpha_r^{(1)}(1;q),\beta_r^{(1)}(1;q))$ is the Bailey pair with
\begin{align}
\alpha_r^{(1)}(1;q)=\left\{\begin{array}{ll}
1, & r=0, \\
2q^{r^2}, & r\geq 1.
\end{array}\right.
\end{align}
Applying \eqref{eq-BP-lift} to it, we obtain a Bailey pair relative to $q$: for $n\geq 0$,
\begin{align}
\alpha_n^{(2)}(q;q)&=\frac{(1-q^{2n+1})q^{n^2}}{1-q}\sum_{r=0}^n q^{-r^2}\alpha_r(1;q)=(2n+1)\frac{(1-q^{2n+1})q^{n^2}}{1-q}, \\
\beta_n^{(2)}(q;q)&=\beta_n^{(1)}(1;q).
\end{align}
Substituting this Bailey pair into \eqref{BP-S1} and iterating $i$ times, we obtain the Bailey pairs $(\alpha_n^{(2+i)}(q;q),\beta_n^{(2+i)}(q;q))$ with
\begin{align}
\alpha_n^{(2+i)}(q;q)=q^{i(n^2+n)}\alpha_n^{(2)}(q;q), \quad i=1,2,\dots, \lambda.
\end{align}
Substituting these Bailey pairs into \eqref{S2-proof-start}, we deduce that
\begin{align}\label{S2-proof-mid}
q^{-c_{\lambda,k}}T_1(q)=\sum_{s_1\geq s_2\geq \cdots \geq s_{k-\lambda-1}\geq 0} \frac{q^{s_1^2+s_2^2+\cdots +s_{k-\lambda-1}^2}}{(q;q)_{s_1-s_2}\cdots (q;q)_{s_{k-\lambda-2}-s_{k-\lambda-1}}}\beta_{s_{k-\lambda-1}}^{(\lambda+2)}(q;q).
\end{align}
Next we reduce the parameter from $q$ to $1$. Applying \eqref{eq-BP-reduce} to the Bailey pair $(\alpha_n^{(\lambda+2)}(q;q),\beta_n^{(\lambda+2)}(q;q))$ we obtain a Bailey pair relative to $1$:
\begin{align}
\alpha_n^{(\lambda+3)}(1;q)&=1, \nonumber  \\
\alpha_n^{(\lambda+3)}(1;q)&=(2n+1)q^{(\lambda+1)n^2+\lambda n}-(2n-1)q^{(\lambda+1)(n-1)^2+\lambda (n-1)+2n-1},  \quad n\geq 1, \nonumber \\
\beta_n^{(\lambda+3)}(1;q)&=\beta_n^{(\lambda+2)}(q;q).
\end{align}
Using \eqref{BP-S1} to iterate this Bailey pair $i$ times, we obtain the Bailey pairs $(\alpha_n^{(\lambda+3+i)}(1;q),$ $\beta_n^{(\lambda+3+i)}(1;q))$ with
\begin{align}
\alpha_n^{(\lambda+3+i)}(1;q)=q^{in^2}\alpha_n^{(\lambda+3)}(1;q), \quad i=1,2,\dots,k-\lambda-2.
\end{align}
Substituting these Bailey pairs into \eqref{S2-proof-mid}, we deduce that
\begin{align}
&q^{-c_{\lambda,k}}T_1(q)=\sum_{s_1\geq s_2\geq \cdots \geq s_{k-\lambda-1}\geq 0} \frac{q^{s_1^2+s_2^2+\cdots +s_{k-\lambda-1}^2}}{(q;q)_{s_1-s_2}\cdots (q;q)_{s_{k-\lambda-2}-s_{k-\lambda-1}}}\beta_{s_{k-\lambda-1}}^{(\lambda+3)}(1;q) \nonumber \\
&=\cdots=\sum_{s_1\geq 0} q^{s_1^2}\beta_{s_1}^{(k+1)}(1;q)\nonumber \\
&=\frac{1}{(q;q)_\infty}\sum_{n=0}^\infty q^{n^2}\alpha_n^{(k+1)}(1;q)  \quad \text{(by \eqref{eq-BP-id-key})} \nonumber \\
&=\frac{1}{(q;q)_\infty} \Big(1+\sum_{n=1}^\infty \Big((2n+1)q^{kn^2+\lambda n}-(2n-1)q^{kn^2-\lambda n}\Big) \Big) \nonumber \\
&=\frac{1}{(q;q)_\infty} \sum_{n=-\infty}^\infty (2n+1)q^{kn^2+\lambda n}.
\end{align}
Here we replaced $n$ by $-n$ in the second sum to get the last equality.
This proves \eqref{id-2}.

In order to prove \eqref{id-3}, we denote its left side as $T_2(q)$ and introduce the variables $s_1,s_2,\dots,s_{k-1}$ and $r$ as in \eqref{eq-variable}, \eqref{eq-variable-relation} and \eqref{odd-nk}. We have
\begin{align}\label{S3-proof-start}
&q^{-c_{\lambda,k}}T_2(q) \nonumber \\
&=q^{\frac{1}{4}(k+2\lambda)}\sum_{s_1\geq s_2\geq \cdots \geq s_{k-1}\geq 0} \frac{q^{s_1^2+\cdots +s_{k-2}^2+s_{k-1}^2+s_1+s_2+\cdots +s_{k-\lambda-1}+2s_{k-\lambda}+\cdots +2s_{k-2}+2s_{k-1}}}{(q;q)_{s_1-s_2}\cdots (q;q)_{s_{k-2}-s_{k-1}}} \nonumber \\
 &\quad \quad \times \sum_{r=-s_{k-1}-1}^{s_{k-1}} \frac{q^{r^2+r}}{(q;q)_{s_{k-1}-r}(q;q)_{s_{k-1}+r+1}} \nonumber \\
&=2q^{\frac{1}{4}(k+2\lambda)}\sum_{s_1\geq s_2\geq \cdots \geq s_{k-1}\geq 0} \frac{q^{s_1^2+\cdots +s_{k-2}^2+s_{k-1}^2+s_1+s_2+\cdots +s_{k-\lambda-1}+2s_{k-\lambda}+\cdots +2s_{k-2}+2s_{k-1}}}{(q;q)_{s_1-s_2}\cdots (q;q)_{s_{k-2}-s_{k-1}}} \nonumber \\
 &\quad \quad \times \sum_{r=0}^{s_{k-1}} \frac{q^{r^2+r}}{(q;q)_{s_{k-1}-r}(q;q)_{s_{k-1}+r+1}} \nonumber \\
 &=\frac{2}{1-q}q^{\frac{1}{4}(k+2\lambda)}\sum_{s_1\geq s_2\geq \cdots \geq s_{k-1}\geq 0} \frac{q^{s_1^2+\cdots +s_{k-2}^2+s_{k-1}^2+s_1+s_2+\cdots +s_{k-\lambda-1}+2s_{k-\lambda}+\cdots +2s_{k-2}+2s_{k-1}}}{(q;q)_{s_1-s_2}\cdots (q;q)_{s_{k-2}-s_{k-1}}} \nonumber \\
 &\qquad \qquad \qquad \qquad \times  \beta_{s_{k-1}}^{(1)}(q;q),
\end{align}
where $(\alpha_r^{(1)}(q;q),\beta_r^{(1)}(q;q))$ is a Bailey pair relative to $q$ with
\begin{align}
\alpha_r^{(1)}(q;q)=q^{r^2+r}, \quad r\geq 0.
\end{align}
Applying \eqref{eq-BP-lift} to this Bailey pair we obtain a Bailey pair relative to $q^2$:
\begin{equation}
\begin{split}
\alpha_n^{(2)}(q^2;q)&=\frac{1-q^{2n+2}}{1-q^2}(n+1)q^{n^2+n}, \\
\beta_n^{(2)}(q^2;q)&=\beta_n^{(1)}(q;q).
\end{split}
\end{equation}
Substituting this Bailey pair into \eqref{BP-S1} and iterating $i$ times, we obtain the Bailey pairs $(\alpha_n^{(2+i)}(q^2;q),\beta_n^{(2+i)}(q^2;q))$ with
\begin{align}\label{S3-alpha-pre}
\alpha_n^{(2+i)}(q^2;q)=q^{i(n^2+2n)}\alpha_n^{(2)}(q^2;q), \quad i=1,2,\dots,\lambda.
\end{align}
From \eqref{S3-proof-start} we have
\begin{align}\label{S3-proof-mid}
&q^{-c_{\lambda,k}}T_2(q)  \\
&=\frac{2}{1-q}q^{\frac{1}{4}(k+2\lambda)}\sum_{s_1\geq s_2\geq \cdots \geq s_{k-2}\geq 0} \frac{q^{s_1^2+\cdots +s_{k-2}^2+s_1+s_2+\cdots +s_{k-\lambda-1}+2s_{k-\lambda}+\cdots +2s_{k-2}}}{(q;q)_{s_1-s_2}\cdots (q;q)_{s_{k-3}-s_{k-2}}} \beta_{s_{k-2}}^{(3)}(q^2;q)\nonumber \\
&=\cdots \nonumber \\
&=\frac{2}{1-q}q^{\frac{1}{4}(k+2\lambda)}\sum_{s_1\geq s_2\geq \cdots \geq s_{k-\lambda-1}\geq 0} \frac{q^{s_1^2+\cdots +s_{k-\lambda-1}^2+s_1+s_2+\cdots +s_{k-\lambda-1}}}{(q;q)_{s_1-s_2}\cdots (q;q)_{s_{k-\lambda-2}-s_{k-\lambda-1}}}\beta_{s_{k-\lambda-1}}^{(\lambda+2)}(q^2;q). \nonumber
\end{align}
Next we apply \eqref{eq-BP-reduce} to the Bailey pair $(\alpha_n^{(\lambda+2)},\beta_n^{(\lambda+2)})$ to obtain the following Bailey pair relative to $q$:
\begin{equation}\label{S3-alpha-mid}
\begin{split}
\alpha_n^{(\lambda+3)}(q;q)&=(n+1)q^{(\lambda+1)n^2+(2\lambda+1)n}-nq^{(\lambda+1)(n-1)^2+(2\lambda+1)(n-1)+2n}, \\
\beta_n^{(\lambda+3)}(q;q)&=\beta_n^{(\lambda+2)}(q^2;q).
\end{split}
\end{equation}
Iterating this Bailey pair using \eqref{BP-S1} $i$ times, we obtain the Bailey pairs $(\alpha_n^{(\lambda+3+i)}(q;q),$ $\beta_n^{(\lambda+3+i)}(q;q))$  with
\begin{align}
\alpha_n^{(\lambda+3+i)}(q;q)=q^{i(n^2+n)}\alpha_n^{(\lambda+3)}(q;q), \quad i=1,2,\dots,k-\lambda-2.
\end{align}
Substituting these Bailey pairs into \eqref{S3-proof-mid}, we deduce that
\begin{align}
&q^{-c_{\lambda,k}}T_2(q)\nonumber\\
&=\frac{2}{1-q}q^{\frac{1}{4}(k+2\lambda)}\sum_{s_1\geq s_2\geq \cdots \geq s_{k-\lambda-1}\geq 0} \frac{q^{s_1^2+\cdots +s_{k-\lambda-1}^2+s_1+s_2+\cdots +s_{k-\lambda-1}}}{(q;q)_{s_1-s_2}\cdots (q;q)_{s_{k-\lambda-2}-s_{k-\lambda-1}}}\beta_{s_{k-\lambda-1}}^{(\lambda+3)}(q;q)\nonumber \\
&=\frac{2}{1-q}q^{\frac{1}{4}(k+2\lambda)}\sum_{s_1\geq s_2\geq \cdots \geq s_{k-\lambda-2}\geq 0} \frac{q^{s_1^2+\cdots +s_{k-\lambda-2}^2+s_1+s_2+\cdots +s_{k-\lambda-2}}}{(q;q)_{s_1-s_2}\cdots (q;q)_{s_{k-\lambda-3}-s_{k-\lambda-2}}}\beta_{s_{k-\lambda-2}}^{(\lambda+4)}(q;q)\nonumber \\
&=\cdots \nonumber \\
&=\frac{2}{1-q}q^{\frac{1}{4}(k+2\lambda)}\sum_{s_1\geq 0} q^{s_1^2+s_1}\beta_{s_1}^{(k+1)}(q;q) \nonumber \\
&=\frac{2}{(q;q)_\infty}q^{\frac{1}{4}(k+2\lambda)} \sum_{n\geq 0} q^{n^2+n}\alpha_n^{(k+1)}(q;q) \quad \text{(by \eqref{eq-BP-id-key})} \nonumber \\
&=\frac{2}{(q;q)_\infty}q^{\frac{1}{4}(k+2\lambda)} \Big(\sum_{n=0}^\infty (n+1)q^{k(n^2+n)+\lambda n}-\sum_{n=0}^\infty nq^{kn^2+(k-\lambda)n-\lambda}  \Big)  \nonumber \\
&=\frac{2}{(q;q)_\infty}q^{\frac{1}{4}(k+2\lambda)}\sum_{n=-\infty}^\infty nq^{kn^2-(k-\lambda)n-\lambda}.
\end{align}
Here for the penultimate equality we used  \eqref{S3-alpha-pre} and \eqref{S3-alpha-mid},  and for the last equality we replaced $n$ by $n-1$ (resp.\ $-n$) in the first (resp.\ second) sum in the penultimate line.  This proves \eqref{id-4}.
\end{proof}

\section{Concluding remarks}\label{sec-remark}
We have proved in Theorem \ref{thm-main} that the Nahm sum $f_{2\mathcal{C}(D_k)^{-1},0,-1/24}(q)$ is modular. Inspired by Zagier's duality expectation (see \eqref{eq-dual}), we propose the following
\begin{conj}
The Nahm sum $f_{\frac{1}{2}\mathcal{C}(D_k),0,(1-k)/24}(q)$ is modular for any $k\geq 3$.
\end{conj}
This was proved by Cao and Wang  for $k=3$ and $k=4$ in \cite[Theorem 6.3]{Cao-Wang2024} and \cite[Theorem 3.1]{Cao-Wang2025}, respectively. For any positive integer $m$ and $1\leq a<m$ we denote
\begin{align}
J_m:=(q^m;q^m)_\infty, \quad J_{a,m}=(q^a,q^{m-a},q^m;q^m)_\infty.
\end{align}
The modular representations for the corresponding Nahm sums are
\begin{align}
&f_{\frac{1}{2}\mathcal{C}(D_3),0,-1/12}(q^2)=q^{-\frac{1}{6}}\sum_{n_1,n_2,n_3\geq 0} \frac{q^{n_1^2+n_2^2+n_3^2-n_1n_2-n_1n_3}}{(q^2;q^2)_{n_1}(q^2;q^2)_{n_2}(q^2;q^2)_{n_3}}\nonumber \\
&=q^{-\frac{1}{6}}\Big(\frac{J_2^3J_6^5}{J_1^2J_3^2J_4^2J_{12}^2}+4q\frac{J_4^2J_{12}^2}{J_2^3J_6}\Big), \\
&f_{\frac{1}{2}\mathcal{C}(D_4),0,-1/8}(q^2)=q^{-\frac{1}{4}}\sum_{n_1,n_2,n_3,n_4\geq 0} \frac{q^{n_1^2+n_2^2+n_3^2+n_4^2-n_1n_4-n_2n_4-n_3n_4}}{(q^2;q^2)_{n_1}(q^2;q^2)_{n_2}(q^2;q^2)_{n_3}(q^2;q^2)_{n_4}} \nonumber \\
&=q^{-\frac{1}{4}}\Big(\frac{J_2^5J_4}{J_1^4J_8^2}+8q\frac{J_4^3J_8^2}{J_2^5}\Big).
\end{align}
Some companion modular Nahm sums associated with the same matrices but different vectors have also been found there.

For $k\geq 3$, it is natural to ask whether $f_{\mathcal{C}(D_k),0,C}(q)$  and its dual Nahm sum $f_{\mathcal{C}(D_k)^{-1},0,-k/24-C}(q)$ are modular for a suitable scalar $C$ or not. In the case $k=3$, we have $D_3=A_3$, and the modularity has been confirmed by Cao and Wang \cite[Theorem 4.1]{Cao-Wang2024} by the identities
\begin{align}
    &f_{\mathcal{C}(D_3),0,-1/14}(q^2)=q^{-\frac{1}{7}}\sum_{n_1,n_2,n_3\geq 0} \frac{q^{2n_1^2+2n_2^2+2n_3^2-2n_1n_2-2n_1n_3}}{(q^2;q^2)_{n_1}(q^2;q^2)_{n_2}(q^2;q^2)_{n_3}} \nonumber  \\
&=q^{-\frac{1}{7}}\Big(\frac{1}{2}\frac{J_{2}^{11}J_{1,14}J_{12,28}}{J_1^{6}J_{4}^{6}J_{28}}
+\frac{1}{2}\frac{J_{1}^{5}J_{7}J_{3,14}J_{5,14}}{J_{2}^{5}J_{4,14}J_{6,14}J_{14}}
-4q^2 \frac{J_4^{6}J_{28}^{3}}{J_2^{6}J_{4,28}J_{10,28}J_{12,28}}\Big), \label{D3-id} \\
&f_{\mathcal{C}(D_3)^{-1},0,-3/56}(q^8)=q^{-\frac{3}{7}}\sum_{n_1,n_2,n_3\geq 0} \frac{q^{4n_1^2+3n_2^2+3n_3^2+4n_1n_2+4n_1n_3+2n_2n_3}}{(q^8;q^8)_{n_1}(q^8;q^8)_{n_2}(q^8;q^8)_{n_3}}\nonumber \\
&=q^{-\frac{3}{7}}\Big(\frac{J_8J_{56}^2J_{112}J_{24,112}J_{40,112}}{J_4J_{12,112}J_{28,112}J_{32,112}J_{44,112}J_{48,112}}+2q^3\frac{J_{16}J_{32,112}}{J_8^2}\Big).  \label{D3-inverse-id}
\end{align}
Here the identity \eqref{D3-id} is \cite[Eq.~(4.5)]{Cao-Wang2024}. Regarding the second identity, recall that Cao and Wang \cite[Eqs.~(4.2) and (4.3)]{Cao-Wang2024} observed that $\mathcal{C}(D_3)$ can be lifted from the matrix $A=\left(\begin{smallmatrix}
    1 & -1/2 \\ -1/2 & 3/4
\end{smallmatrix}\right)$ in Zagier's ninth rank two example \cite[Table 2]{Zagier}. Hence,
\begin{align}
f_{\mathcal{C}(D_3)^{-1},0,0}(q^8)=f_{A,0,0}(q^8).
\end{align}
Therefore, we can use the product form  \cite[Eq.~(3.123)]{Wang-rank2} of the rank two Nahm sum on the  right side to obtain \eqref{D3-inverse-id}. However, the case $k\geq 4$ remains open.

\subsection*{Acknowledgements}
This work was supported by the National Key R\&D Program of China (Grant No.\ 2024YFA1014500). The authors thank Ole Warnaar for bringing \cite{Warnaar} to our attention and some helpful comments.


\begin{thebibliography}{0}
\bibitem{Andrews1974} G.E. Andrews, An analytic generalization of the Rogers--Ramanujan identities for odd moduli, Proc. Natl. Acad. Sci. USA 71(10) (1974), 4082--4085.

\bibitem{Andrews1981} G.E. Andrews, Multiple $q$-series identities, Houston J. Math. 7(1) (1981), 11--22.

\bibitem{Andrews1998} G.E. Andrews, The Theory of Partitions, Addison-Wesley, 1976; Reissued Cambridge, 1998.



\bibitem{Bressoud2000} D.M. Bressoud, M.E.H. Ismail and D. Stanton, Change of base in Bailey pairs, Ramanujan J. 4 (2000), 435--453.



\bibitem{CMP} C. Calinescu, A. Milas and M. Penn, Vertex algebraic structure of principal subspaces of basic $A_{2n}^{(2)}$-modules, J. Pure Appl. Algebra 220 (2016), 1752--1784.

\bibitem{CRW} Z. Cao, H. Rosengren and L. Wang, On some double Nahm sums of Zagier, J. Combin. Theory Ser. A 202 (2024), Paper No. 105819.

\bibitem{Cao-Wang2024} Z. Cao and L. Wang, Some new modular rank three Nahm sums from a lift-dual operation, arXiv:2412.15767.

\bibitem{Cao-Wang2025} Z. Cao and L. Wang, Some new modular rank four Nahm sums as lift-dual of rank three examples, arXiv:2508.12468.

\bibitem{Feigin} I. Cherednik and B. Feigin, Rogers--Ramanujan type identities and Nil-DAHA, Adv. Math. 248 (2013), 1050--1088.



\bibitem{FGK} M. Flohr, C. Grabow and M. Koehn, Fermionic expressions for the characters of $c_{p,1}$ logarithmic conformal
field theories, Nucl. Phys. B 768 (2007), 263--276.


\bibitem{Kedem} R. Kedem, T.R. Klassen, B.M. McCoy and E. Melzer, Fermionic quasi-particle representations for characters of $(G^{(1)})_1\times (G^{(1)})_1/(G^{(1)})_2$, Physics Letters B 304 (1993), 263--270.


\bibitem{Lovejoy2004} J. Lovejoy, A Bailey lattice, Proc. Am. Math. Soc. 132 (2004), 1507--1516.


\bibitem{Melzer} E. Melzer, Supersymmetric analogs of the Gordon--Andrews identities, and related TBA systems, arXiv:hep-th/9412154.



\bibitem{MW24} A. Milas and L. Wang, Modularity of Nahm sums for the tadpole diagram, Int. J. Number Theory 20 (1) (2024), 73--101.

\bibitem{Nahm1994} W. Nahm, Conformal field theory and the dilogarithm, In 11th International Conference on Mathematical Physics (ICMP-11) (Satelite colloquia: New Problems in General Theory of Fields and Particles), Paris, 1994, 662--667.


\bibitem{Nahm07} W. Nahm, Conformal field theory and torsion elements of the Bloch group. In Frontiers in Number Theory, Physics, and Geometry II: On Conformal Field Theories, Discrete Groups and Renormalization, pp.\ 67--132. Springer, 2007.


\bibitem{Rogers} L.J. Rogers, Second memoir on the expansion of certain infinite products, Proc. London Math. Soc. 25 (1894), 318--343.


\bibitem{Shi-Wang} C. Shi and L. Wang, Modularity of tadpole Nahm sums in ranks 4 and 5, arXiv:2504.17737.

\bibitem{Stembridge} John R. Stembridge, Hall--Littlewood functions, plane partitions, and the Rogers--Ramanujan identities, Trans. Amer. Math. Soc. 319 (1990), no. 2, 469--498.


\bibitem{VZ}M. Vlasenko and S. Zwegers, Nahm's conjecture: asymptotic computations and counterexamples, Commun. Number Theory Phys. 5(3) (2011), 617--642.



\bibitem{Wang-rank2} L. Wang, Identities on Zagier's rank two examples for Nahm's problem, Res. Math. Sci. 11, 49 (2024).

\bibitem{Wang-rank3}L. Wang, Explicit forms and proofs of Zagier's rank three examples for Nahm's problem, Adv. Math. 450 (2024), 109743.


\bibitem{Warnaar-2001} S.O. Warnaar, 50 years of Bailey's lemma, In: Betten, A., Kohnert, A., Laue, R., Wassermann, A. (eds) Algebraic Combinatorics and Applications. Springer, Berlin, Heidelberg, 2001.


\bibitem{Warnaar-2003} S.O. Warnaar, The generalized Borwein conjecture. II. Refined $q$-trinomial coefficients. Discrete Math. 272 (2-3) (2003), 215--258.

\bibitem{Warnaar} S.O. Warnaar, Proof of the Flohr--Grabow--Koehn conjectures for characters of logarithmic conformal field theory, J. Phys. A: Math. Theor. 40 (2007), 12243--12254.


\bibitem{Zagier} D. Zagier, The dilogarithm function, in Frontiers in Number Theory, Physics and Geometry, II, Springer, 2007, 3--65.

\end{thebibliography}
\end{document}